\documentclass{amsart}
\usepackage[margin=1in]{geometry}
\usepackage{amsmath}
\usepackage{amsfonts}
\usepackage{amssymb}
\usepackage{amsthm}

\newcommand{\CC}{\mathbb{C}}

\newcommand{\ZZ}{\mathbb{Z}}
\newcommand{\QQ}{\mathbb{Q}}

\theoremstyle{theorem}
\newtheorem{theorem}{Theorem}
\newtheorem{lemma}{Lemma}

\begin{document}
\title{A Corrigendum to Unreasonable Slightness}

\author{Arseniy Sheydvasser}

\maketitle

\begin{abstract}
We revisit Bogdan Nica's 2011 paper, \emph{The Unreasonable Slightness of $E_2$ over Imaginary Quadratic Rings}, and correct an inaccuracy in his proof.
\end{abstract}

\section{Introduction:}
We review briefly the setup of Nica's paper \cite{nica11}. Let $A$ be a commutative ring, and $E_n(A)$ the group generated by elementary matrices, i.e. matrices in $SL_n(A)$ that have ones on the diagonal, and have one non-zero off-diagonal entry. It is a basic fact that $SL_n(\ZZ) = E_n(\ZZ)$, which prompts the question: what happens for other commutative rings $A$?

In the course of their resolution of the congruence subgroup problem, Bass, Milnor, and Serre \cite{BMS} settled the matter for algebraic number fields in higher dimensions:

	\begin{theorem}[Bass-Milnor-Serre]
	Let $A = \mathcal{O}_K$ be the ring of integers of any algebraic number field $K$. Then, if $n > 2$, $SL_n(A) = E_n(A)$.
	\end{theorem}
	
Similarly, in dimension 2 and for all algebraic number fields other than the imaginary quadratic ones, Vaserstein \cite{vaserstein} showed that elementary matrices are enough:

	\begin{theorem}[Vaserstein]
	Let $A = \mathcal{O}_K$ be the ring of integers of an algebraic number field $K$. Then, if $K$ is not imaginary quadratic, $SL_2(A) = E_2(A)$.
	\end{theorem}

In contrast, imaginary quadratic fields in dimension 2 behave very differently, as was first shown by Cohn \cite{cohn}:

	\begin{theorem}[Cohn]
	Let $K$ be an imaginary quadratic field, and let $A = \mathcal{O}_K$ be its ring of integers. If $K = \QQ(i),\QQ(\sqrt{-2}), \QQ(\sqrt{-3}), \QQ(\sqrt{-7}), \QQ(\sqrt{-11})$, then $SL_2(A) = E_2(A)$. Otherwise, $SL_2(A) \neq E_2(A)$.
	\end{theorem}

We examine more closely the case where $A$ is an imaginary quadratic ring: that is, a subring of $\CC$ of the form $\ZZ[\omega]$, where $\omega$ is a non-real algebraic integer of degree 2. $A$ is necessarily of the form $\ZZ[\sqrt{-D}]$ or $\ZZ[\frac{1}{2}(1 + \sqrt{1 - 4D})]$ for some integer $D \geq 1$. Nica then gives an elementary proof of the following fact:

	\begin{theorem}[Nica]
	Let $A = \ZZ[\sqrt{-D}]$ or $A = \ZZ[\frac{1}{2}(1 + \sqrt{1 - 4D})]$ with $D \geq 4$. Then $E_2(A)$ is an infinite index, non-normal subgroup of $SL_2(A)$.
	\end{theorem}
	
Nica's proof hinges on the existence of non-trivial solutions to the Pell equation $X^2 - D Y^2 = 1$. Since non-trivial solutions exist if and only if $D$ is not a perfect square, the proof is incomplete in the case of rings $\ZZ[di]$ where $d \in \ZZ$. The basic method is still sound, however, and we will provide a simple correction to cover this case as well.

\subsection*{Acknowledgements} The author is much indebted to Alex Kontorovich, both for pointing out the flaw in Nica's paper, and for his helpful suggestions along the way.

\section{Preliminaries:}
Let $\mathcal{U}_2(A)$ denote the set of unimodular pairs; that is, pairs $(\alpha, \beta) \in A^2$ that are the top row of some matrix $\left(\begin{smallmatrix} \alpha & \beta \\ \ast & \ast \end{smallmatrix}\right) \in SL_2(A)$. The usefulness of unimodular pairs is captured in the following lemma:

	\begin{lemma}\label{The Duh Lemma}
	Let $A$ be a commutative ring. Let $L_2(A)$ be the subgroup of $SL_2(A)$ of elements of the form $\left(\begin{smallmatrix} 1 & \\ \ast & 1 \end{smallmatrix}\right)$. Then there exists a bijective map:
	
		\begin{align*}
		L_2(A) \backslash SL_2(A) / E_2(A) & \rightarrow \mathcal{U}_2(A) / E_2(A) \\
		\left[\begin{pmatrix} \alpha & \beta \\ \gamma & \delta \end{pmatrix}\right] & \mapsto \left[(\alpha, \beta)\right].
		\end{align*}
	\end{lemma}
	
	\begin{proof}
	We must first check that this map is well-defined. In particular, we should check that multiplication on the left by an element of $L_2(A)$ does not affect which orbit $\left(\begin{smallmatrix} \alpha & \beta \\ \gamma & \delta \end{smallmatrix}\right)$ maps to. Indeed:
	
		\begin{align*}
		\begin{pmatrix} 1 & \\ \ast & 1 \end{pmatrix} \begin{pmatrix} \alpha & \beta \\ \gamma & \delta \end{pmatrix} = \begin{pmatrix} \alpha & \beta \\ \ast \alpha + \gamma & \ast \beta + \delta \end{pmatrix} \mapsto (\alpha, \beta).
		\end{align*}
		
	That the map is surjective is clear from the definition of a unimodular pair. It remains to prove injectivity. Fix an element $[(\alpha, \beta)] \in \mathcal{U}_2(A) / E_2(A)$, and suppose that it has two pre-images, for which we choose explicit coset representatives $\left(\begin{smallmatrix} \alpha & \beta \\ \gamma & \delta \end{smallmatrix}\right)$ and $\left(\begin{smallmatrix} \alpha & \beta \\ \gamma' & \delta' \end{smallmatrix}\right)$. But then:
	
		\begin{align*}
		\begin{pmatrix} \alpha & \beta \\ \gamma & \delta \end{pmatrix} \begin{pmatrix} \alpha & \beta \\ \gamma' & \delta' \end{pmatrix}^{-1} &=\begin{pmatrix} \alpha & \beta \\ \gamma & \delta \end{pmatrix} \begin{pmatrix} \delta' & -\beta \\ -\gamma' & \alpha \end{pmatrix} = \begin{pmatrix} 1 & \\ \ast & 1 \end{pmatrix} \in L_2(A).
		\end{align*}
		
		Therefore:
		
		\begin{align*}
		\left[\begin{pmatrix} \alpha & \beta \\ \gamma & \delta \end{pmatrix}\right] &= \left[\begin{pmatrix} \alpha & \beta \\ \gamma' & \delta' \end{pmatrix}\right].
		\end{align*}
	\end{proof}
	
Notice that $L_2(A) \subset E_2(A)$. Lemma ~\ref{The Duh Lemma} tells us two important facts. Firstly, it tells us that if we want to know the size of $SL_2(A) / E_2(A)$, we can get at this information by just considering the orbits of unimodular pairs---in particular, $E_2(A)$ is an infinite index subgroup if $U_2(A)/E_2(A)$ is an infinite set. Secondly, it gives us a convenient way to examine the normality of $E_2(A)$---if $E_2(A)$ is a normal subgroup, then it is an immediate consequence that the obvious map $SL_2(A)/E_2(A) \rightarrow \mathcal{U}_2(A)/E_2(A)$ is a bijective map, and then we can transfer the group structure on $SL_2(A)/E_2(A)$ to $\mathcal{U}_2(A)/E_2(A)$. As we shall see, this will significantly simplify computations, and give an easy way to prove that $E_2(A)$ is non-normal.

In order to prove the infinitude of orbits of $\mathcal{U}_2(A)$, Nica introduced the concept of special pairs, defined as all unimodular pairs $(\alpha, \beta)$ such that $|\alpha| = |\beta| < |\alpha \pm \beta|$. He then proved the key result:

	\begin{lemma}\label{Nica's lemma}
	Let $A = \ZZ[\sqrt{-D}]$ or $A = \ZZ[\frac{1}{2}(1 + \sqrt{1 - 4D})]$ with $D \geq 4$. Let $(\alpha, \beta)$, $(\alpha', \beta')$ be special pairs. Then $(\alpha', \beta')$ is $E_2(A)$-equivalent to $(\alpha, \beta)$ iff $(\alpha',\beta') = (\alpha, \beta)$, $(\beta, -\alpha)$, $(-\alpha, -\beta)$, or $(-\beta, \alpha)$.
	\end{lemma}
	
With this lemma, proving $E_2(A)$ is infinite index reduces to constructing an infinite family of special pairs---since any special pair can only be $E_2(A)$-equivalent to finitely many other special pairs, we know immediately that $\mathcal{U}_2(A)/E_2(A)$ is infinite, which gives the desired result. For non-normality, one uses the arithmetic of the matrix group $SL_2(A)$ to derive a contradiction assuming $E_2(A)$ is normal (although Nica does not make use of Lemma ~\ref{The Duh Lemma}, his proof is easily seen to be equivalent to this strategy). Better yet, it is really enough to do the construction for rings $A = \ZZ[\sqrt{-D}]$, as taking $D' = 4D - 1$, it is immediate that this construction works just as well for rings $A = \ZZ[\frac{1}{2}(1 + \sqrt{1 - 4D})]$.

Nica's construction of such an infinite family hinges crucially on the existence of solutions to the Pell equation $X^2 - DY^2 = 1$. As mentioned previously, this makes it unsuitable for the rings $\ZZ[di]$.

\section{The Case $A = \ZZ[\MakeLowercase{di}]$:}

Thankfully, this is easily remedied by the construction of a new infinite family of special unimodular pairs.

\begin{proof}
Specifically, notice that if $d | n$, then:

	\begin{align*}
	\begin{pmatrix} 1 + n + ni & 1 + n - ni \\ n & 1 - n i \end{pmatrix} \in SL_2\left(\ZZ[di]\right).
	\end{align*}
	
Thus the pair $(1 + n + ni, 1 + n - ni)$ is unimodular. That it is special is straightforward, since if $n > 1$:

	\begin{align*}
	\|1 + n + ni\|^2 &= (1 + n)^2 + n^2 = \|1 + n - ni\|^2 \\
	\|1 + n + ni + (1 + n - ni)\|^2 &= (2 + 2n)^2 > (1 + n)^2 + n^2 \\
	\|1 + n + ni - (1 + n - ni)\|^2 &= (2n)^2 > (1 + n)^2 + n^2.
	\end{align*}
	
Therefore, for any $d \geq 2$, we have an infinite family of special unimodular pairs. Applying Lemma ~\ref{Nica's lemma}, we see immediately that $E_2(A)$ is an infinite index subgroup of $SL_2(A)$ by the above discussion.

Proving that $E_2(A)$ is a non-normal subgroup is similarly easy. Indeed, suppose that $E_2(A)$ is normal, so that $\mathcal{U}_2(A)/E_2(A)$ inherits a group structure. We have matrix relations:

	\begin{align*}
	\begin{pmatrix} 0 & 1 \\ -1 & 0 \end{pmatrix} \begin{pmatrix} 1 - ni & -n \\ -n & 1 + ni \end{pmatrix} \begin{pmatrix} 0 & -1 \\ 1 & 0 \end{pmatrix} &= \begin{pmatrix} 1 + ni & n \\ n & 1 - ni \end{pmatrix} \\	
	\begin{pmatrix} 1 - ni & -n \\ -n & 1 + ni \end{pmatrix} &= \begin{pmatrix} 1 + ni & n \\ n & 1 - ni \end{pmatrix}^{-1}.
	\end{align*}
	
Projecting into $\mathcal{U}_2(A)/E_2(A)$, this simplifies greatly (especially since it is readily checked that $\left(\begin{smallmatrix} 0 & 1 \\ -1 & 0 \end{smallmatrix}\right) \in E_2(A)$):

	\begin{align*}
	[(1 - ni, -n)] &= [(1 + ni, n)] \\
	[(1 - ni, -n)] &= [(1 + ni, n)]^{-1}.
	\end{align*}
	
In particular, this implies that $[(1 - ni, -n)]^2 = 1$. However, it is also true that:
	
\begin{align*}
\begin{array}{ll}
	\begin{pmatrix} 1 - ni & -n \\ -n & 1 + ni \end{pmatrix}^2 &= \begin{pmatrix} 1 - 2ni	& -2n \\ -2n & 1 + 2ni \end{pmatrix} \vspace{5mm}\\
  \multicolumn{2}{l}{\begin{pmatrix} 0 & -1 \\ 1 & 0 \end{pmatrix} \begin{pmatrix} 1 - 2ni & -2n \\ -1 - 2n + 2ni & 1 + 2n + 2ni \end{pmatrix} \begin{pmatrix} 0 & 1 \\ -1 & 0 \end{pmatrix}} \\
	&=\begin{pmatrix} 1 + 2n + 2ni & 1 + 2n - 2ni \\ 2n & 1 - 2ni \end{pmatrix},
\end{array}
\end{align*}
	
\noindent which projects to:

	\begin{align*}
	[(1 - ni, -n)]^2 &= [(1 - 2ni, -2n)] \\
	[(1 - 2ni, -2n)] &= [(1 + 2n + 2ni, 1 + 2n - 2ni)].
	\end{align*}
	
By Lemma ~\ref{Nica's lemma},
	
	\begin{align*}
	[(1 + 2n + 2ni, 1 + 2n - 2ni)] \neq [(1 + 2n' + 2n'i, 1 + 2n' - 2n'i)],
	\end{align*}
	
\noindent for distinct values of $n, n'$, so we can always choose $n$ such that $[(1 + 2n + 2ni, 1 + 2n - 2ni)] \neq 1$. But by the above, this implies that $[(1 - ni, -n)]^2 \neq 1$, which is a contradiction. Therefore, $E_2(A)$ is a non-normal subgroup of $SL_2(A)$.
\end{proof}

Equivalently, one can show directly that the matrix:

\begin{align*}
\begin{pmatrix} 1 - n i & -n \\ -n & 1 + ni \end{pmatrix}\begin{pmatrix} 0 & 1 \\ -1 & 0 \end{pmatrix}\begin{pmatrix} 1 - n i & -n \\ -n & 1 + ni \end{pmatrix}^{-1}
\end{align*}

\noindent does not belong to $E_2(A)$, proving that $E_2(A)$ is non-normal. We leave this as an exercise for the reader.

\hskip0.1in

\address{Department of Mathematics, Yale University, 10 Hillhouse Avenue, New Haven, CT 06511}

\email{arseniy.sheydvasser@yale.edu}
\end{document}